\documentclass[11pt]{article}

\usepackage[a4paper,margin=2.5cm]{geometry}

\usepackage{amsmath,amssymb,amsfonts,amsthm,mathtools}

\usepackage{graphicx}
\usepackage{float}
\usepackage{booktabs}
\usepackage{multirow}
\usepackage{array}
\usepackage{graphicx}%
\usepackage{multirow}%
\usepackage{amsmath,amssymb,amsfonts}%
\usepackage{amsthm}%
\usepackage{mathrsfs}%
\usepackage[title]{appendix}%
\usepackage{xcolor}%
\usepackage{textcomp}%
\usepackage{manyfoot}%
\usepackage{booktabs}%
\usepackage{algorithm}%
\usepackage{algorithmicx}%
\usepackage{algpseudocode}%
\usepackage{listings}%
\usepackage{soul}
\usepackage{tikz}
\usepackage{enumitem}
\usepackage{setspace}
\usepackage[hidelinks]{hyperref}

\newtheorem{theorem}{Theorem}[section]
\newtheorem{lemma}[theorem]{Lemma}
\newtheorem{proposition}[theorem]{Proposition}
\newtheorem{corollary}[theorem]{Corollary}

\theoremstyle{definition}
\newtheorem{definition}[theorem]{Definition}
\newtheorem{example}[theorem]{Example}
\newtheorem{remark}[theorem]{Remark}

\title{\Large\bfseries
Directed Extended Zero Divisor Graphs Of Non-Commutative Semigroups}

\author{
	Defne Somer$^{1}$,
	Didem Ye\c{s}il$^{1,*}$,
	\.{I}smail Naci Cang\"ul$^{2}$
}

\date{}

\begin{document}
	
	\maketitle
	
	\begin{center}
		\small
		
		$^{1}$Department of Mathematics,\\
		\c{C}anakkale Onsekiz Mart University,\\
		\c{C}anakkale, T\"urkiye
		
		\vspace{0.4cm}
		
		$^{2}$Department of Mathematics,\\
		Bursa Uluda\u{g} University,\\
		Bursa, T\"urkiye
		
		\vspace{0.5cm}
		
		\texttt{defne.somer@comu.edu.tr}
		
		\texttt{dyesil@comu.edu.tr}$^{*}$
		
		\texttt{cangul@uludag.edu.tr}
		
		\vspace{0.3cm}
		
		$^{*}$Corresponding author.
		
	\end{center}
	
	\begin{abstract}
	In this paper, the directed extended zero-divisor graph $\overrightarrow{\Gamma}_E(S)$ of a non-commutative semigroup $S$ is introduced and the relationships between $\overrightarrow{\Gamma}_E(S)$ and the zero-divisor graph $\overrightarrow{\Gamma}(S)$ of $S$ are examined. The fact that $\overrightarrow{\Gamma}(S)$ is a spanning subgraph of $\overrightarrow{\Gamma}_E(S)$ is proved. Necessary and sufficient conditions for the equality $\overrightarrow{\Gamma}_E(S)=\overrightarrow{\Gamma}(S)$ are obtained. Moreover, the conditions under which $\overrightarrow{\Gamma}_E(S)$ contains a cycle were established. In addition, conditions under which the diameter of $\overrightarrow{\Gamma}_E(S)$ is strictly smaller than that of $\overrightarrow{\Gamma}(S)$ are characterized. The vertices that are adjacent to or adjacent from all other vertices are investigated. The results are established by using several algebraic notions such as nilpotent elements, their nilpotency indices, and annihilator sets of elements.
	\end{abstract}
	
	\bigskip
	
	
	\medskip
	
	
\section{Introduction}

In his paper titled ``Coloring of commutative rings", Beck (1988, \cite{beck}) initiated the study on zero-divisor graphs of rings which have zero divisors of $R$ (denoted by $Z(R)$) together with $0$ as their vertex set and two distinct vertices $x,y$ are adjacent if and only if $xy=0$. Given that the definition is based upon commutative rings, the graph defined in Beck's work was undirected and simple. In \cite{beck}, Beck computed chromatic numbers and clique numbers of zero-divisor graphs of commutative rings, and searched for the graph colorings. 

However, aforementioned graph was not named as zero-divisor graph until Anderson and Livingston enhanced the definition in \cite{dfafsl}. In their graph, $R$ is a commutative ring with unity, the set of vertices is $Z(R)-\{0\}$, and distinct $x,y$ are adjacent vertices if and only if $xy=0$.  This definition came to be the most used definition for zero-divisor graphs of commutative rings. Research on zero-divisor graphs of commutative rings expanded with researches such as \cite{demeyer2, vnr, lag, mulay1, redmond1, monius}.

Predictably, the research field did not adhere to zero-divisor graphs of commutative rings with identity. In \cite{redmond}, Redmond proposed multiple approaches for defining the zero-divisor graph of a non-commutative ring. Among these formulations, one defines the zero-divisor graph of a non-commutative ring $R$ as the graph with vertex set $Z(R)-\{0\}$ and $x \to y$ is a directed arc between distinct vertices if and only if $xy=0$. This definition omits loops and multi arcs as well, making the directed zero-divisor graph simple. The works followed this idea are \cite{behboodi, akbari2}. 

In parallel with the development of zero-divisor graphs of rings, the investigation was extended to commutative semigroups in 2002 when \cite{demeyer3} initiated the research on zero-divisor graphs of commutative semigroups, defined analogously to those of rings. The zero-divisor graph of a commutative semigroup $S$ is denoted by $\Gamma(S)$. Subsequent contributions include \cite{demeyer4, wulu, bender}. In \cite{wright}, Wright carried over the concept of the directed zero-divisor graph of rings to semigroups, expressed as $\overrightarrow{\Gamma}(S)$. In addition to $\overrightarrow{\Gamma}(S)$, Wright studied  $\overline{\Gamma}(S)$, the underlying (undirected) graph of $\overrightarrow{\Gamma}(S)$, and $\overrightarrow{\Gamma}_0(S)$, the larger digraph with loops (i.e. for $x \in Z(S)-\{0\}$, $x \to x$ is a loop iff $x^2=0$).

Several variants have also been introduced. For example, as mentioned in \cite{demeyer3}, for a commutative semigroup $S$ and its ideal $I$, the graph $\Gamma((S,I))$ has the elements $s_1,s_2 \in S - I$ such that $s_1 s_2 \in I$. Distinct $s_1$ and $s_2$ are adjacent if $s_1 s_2 \in I$. In \cite{flores}, Beaugris et al. presented the notion of weak zero divisor of finite commutative rings. Vertices are non-zero elements of the ring, while distinct $x,y \in R-\{0\}$ are adjacent if and only if there exists a positive integer $n \in \mathbb{Z}^+$ such that $(xy)^n =0$. Moreover, in \cite{extended}, Bennis et al. defined the extended zero-divisor graph of commutative rings with unity (also denoted by $\overline{\Gamma}(R)$). In the extended zero-divisor graph, the set of vertices is non-zero zero divisors, whereas distinct $x,y$ are adjacent if and only if there exist positive integers $m,n$ such that $x^m y^n = 0$ with both $x^m$ and $y^n$ are non-zero. The zero-divisor graph of a ring $R$ is always the subgraph of its extended zero-divisor graph, \cite{extended}. This particular definition is one of the primary concerns of this paper. Furthermore, in \cite{dey}, the authors proposed a new definition for zero-divisor graphs of a $\Gamma$-semigroup where distinct zero divisors $x,y$ are adjacent if and only if there exists an $\alpha \in \Gamma$ such that $x \alpha y=0$. In their paper, the zero-divisor graph of $\mathbb{Z}_n$ with $\Gamma = \mathbb{Z}_n$ is investigated. The regular zero-divisor graph is a subgraph of this graph. 

Additionally, zero-divisor graphs have been studied not only for rings and semigroups but also for other algebraic structures. Such as \cite{semiring, dolzan} where the authors, in addition to rings, studied the zero-divisor graphs of semirings. Zero-divisor graphs of modules are studied in \cite{baziar}. In \cite{cannon}, Cannon et al. investigated the zero-divisor graphs of nearrings and semigroups.

Furthermore, in addition to zero-divisor graphs, many other graphs associated with algebraic structures have been introduced and extensively studied in the literature. For instance, the commuting graph of a semigroup is a graph whose vertices are the non-central elements of the semigroup and distinct vertices $x,y$ are adjacent if and only if they commute, \cite{ara}. The notion of commuting graphs was originally introduced in the context of groups, \cite{segev}, and has also been studied for rings and semigroups. Moreover, commuting graphs have been investigated for various special classes of semigroups, groups, and rings, \cite{ paulista, rom, torktaz, dorbidi}. 

Another field in which graph theory has found extensive applications is chemistry, where various graph-based topological indices are introduced and investigated to quantify structural properties of graphs. For instance, the atom-bond connectivity energy of a graph has been studied together with several bounds and values for different graph classes, \cite{reddy}. First Zagreb and second Zagreb indices are graph indices that have been studied extensively, \cite{gunes, zagreb}. In 2021, Nikmehr et al. defined weakly zero-divisor graph of a commutative ring, \cite{nikme}. In \cite{rehman}, Rehman et al. computed the Zagreb indices of the weakly zero-divisor graph of the ring $\mathbb{Z}_p \times \mathbb{Z}_t \times \mathbb{Z}_s$ where $p,t,s$ are not-necessarily distinct primes that are greater than 2, and in \cite{semil}, Semil et al. presents some findings related to the first Zagreb index of zero divisor graph for the commutative ring $\mathbb{Z}_{p^k}$, where $p$ is a prime number. These two papers bring together two different aspects of graph theory. On one hand, they involve zero-divisor graphs arising from algebraic structures, and on the other hand, they consider graph indices, which are studied in chemical graph theory. In this area, these indices are usually named as molecular descriptors. In this sense, they establish a connection between algebra and chemistry. 

In this paper, the definition of the extended zero-divisor digraph of a non-commutative semigroup $S$, denoted as $\overrightarrow{\Gamma}_E(S)$, is given analogously to those of commutative rings from \cite{extended}. It has been shown that $\overrightarrow{\Gamma}_E(S)$ is a supergraph of $\overrightarrow{\Gamma}(S)$, the zero-divisor digraph of $S$. The condition under which $\overrightarrow{\Gamma}(S)
= \overrightarrow{\Gamma}_E(S)$ has been examined and characterized. Since the diameter and girth of the supergraph cannot be greater than those of its subgraphs, conditions in which the diameter and girth of $\overrightarrow{\Gamma}_E (S)$ are lesser than those of $\overrightarrow{\Gamma}(S)$ are determined. The effects of nilpotents and annihilator sets of elements on $\overrightarrow{\Gamma}_E(S)$ are studied, particularly their influence on parameters such as connectivity, diameter and girth are determined. 

\section{Preliminaries}\label{prel}
In this section, the definitions used throughout the paper are recalled.
\subsection{Preliminaries on semigroups} 

\begin{definition}\cite{radical}
	Let $S$ be a semigroup and $A \subseteq S$. The set $\sqrt{A} = \{x \in S: x^n \in A, \text{ } \exists n \in \mathbb{Z}^+ \}$ is called the radical of $A$.
\end{definition}

\begin{definition}\cite{howie}
	Let $S$ be a semigroup with $0$ and $0 \neq s \in S$. If there exists a positive integer $n \geq 2$ such that $s^n = 0$, then $s$ is called a nilpotent element. $Nil(S)$ denotes the set of nilpotent elements of $S$.
\end{definition}

\begin{remark}
	Throughout this paper, for an element $s \in Nil(S)$, $n_s$ will denote the smallest positive integer such that $s^{n_s}=0$; that is $s^{n_s - 1}\neq 0$.	
\end{remark}

\begin{definition}\cite{howie}
	Let $S$ be a semigroup and $e \in S$. If $e^2=e$, then $e$ is called an idempotent element. The set of idempotents of $S$ is denoted as $E(S)$.
\end{definition}

\begin{definition} \cite{wright}
	Let $S$ be a semigroup with $0$ and assume $a,b \in S$ are non-zero. If $ab=0$, then $a$ is called a left zero divisor and $b$ is called a right zero divisor. If for a $s \in S$, there exist $a,b \neq 0$ such that $as = 0$ and $sb= 0$, then $s$ is called a two-sided zero divisor. The set $Z(S)$ denotes the set of zero divisors of $S$, namely, if $x \in Z(S)$ then $x$ might be a right, left or a two-sided zero divisor. 
\end{definition}

\begin{remark}
	Let $S$ be a semigroup and $A,B \subseteq S$. Then $AB=\{ab : a\in A, b\in B\}$.
\end{remark}

\begin{definition}\cite{trans}
	Let $S$ be a semigroup with $0$. For a $x \in S$, the set $Ann(x)=\{s \in S:sx = 0 = xs\}$ is called the set of annihilators of $x$.
\end{definition}

\begin{definition}\cite{semiring}
	Let $S$ be a non-commutative semigroup with $0$. For a $x \in S$, the set $Ann_l(x)= \{s \in S:sx=0\}$ is called the set of left annihilators and $Ann_r(x)=\{s \in S:xs=0\}$ is called the set of right annihilators.
\end{definition}

\begin{remark}\label{kap}
	Let $S$ be a non-commutative semigroup with $0$. Then, $Ann_l(x) \cap Ann_r(x) = Ann(x)$.
\end{remark}

\begin{lemma}\label{kapsar}
	Let $S$ be a non-commutative semigroup with $0$, $x \in S - \{0\}$, and $n \in \mathbb{Z}^{+}$. If $x \in Nil(S)$ then, $Ann(x) \subseteq Ann (x^n)$ for all $n \geq 2$.
\end{lemma}

\subsection{Preliminaries on graphs}
\begin{definition}\label{graph} \cite{digraph}
	A directed graph (or a digraph) $D$ consists of two sets, the set of vertices denoted by $V(D)$ and the set of arcs denoted by $E(D)$. Elements of $V(D)$ are called vertices and elements of $E(D)$ are called arcs. $E(D)$ in fact consists of ordered pairs of vertices. 
\end{definition}

\begin{definition} \cite{wright, digraph} 
	Let $D$ be a digraph. If $(u,v)$ is an arc of $D$, then $u$ is said to be adjacent to $v$ and $v$ is adjacent from $u$, and is denoted by $u \to v$. If $e$ is the arc between $u$ and $v$; then $u$ is called the tail of $e$, while $v$ is called the head of $e$.
\end{definition}

\begin{definition}\cite{handbook}
	Let $G$ be an undirected graph and $n \in \mathbb{Z}^+$. A walk 
	\begin{align}\label{walk}
		W = v_0, e_1, v_1, \cdots, e_n, v_n
	\end{align}
	in $G$, is an alternating sequence of edges and vertices and for all $i \in {1,2,\cdots,n}$ the vertices $v_{i-1}$ and $v_{i}$ are ends of the edge $e_i$. 
	
	If $G$ is a digraph, then $P$ is a sequence of edges and vertices where for all $i \in {1,2,\cdots,n}$, the vertex $v_{i-1}$ is the tail of the arc $e_i$  while $v_{i}$ is the head of the arc $e_i$. 
\end{definition}

\begin{definition} \cite{handbook}
	In the walk $W$ given in (\ref{walk}), $v_0$ is called the initial vertex, $v_n$ is called the terminal vertex, and all other vertices are called internal vertices.
\end{definition}

\begin{definition} \cite{handbook}
	Let $W$ be the walk given in (\ref{walk}). If no edge and no internal vertex occurs more than once, then $W$ is called a path. If $v_0 = v_n$, then $W$ is called a cycle.
\end{definition}

\begin{remark}
	Length of a walk (or a path, or a cycle) is equal to the number of occuring edges.
\end{remark}

\begin{definition} \cite{rd}
	Let $G$ be a graph. The length of the shortest cycle is called the girth of $G$ and is denoted as $Girth(G)$. If $G$ does not contain a cycle, then $G$ is called a tree and $Girth(G)=\infty$.
\end{definition}

\begin{definition}\cite{digraph}
	Let $x,y$ be vertices of a digraph $D$. The distance between $x$ and $y$, denoted by $dist(x,y)$, is the length of the shortest directed path from $x$ to $y$.
\end{definition}

\begin{definition}\cite{digraph}
	Diameter of a digraph $D$, denoted by $Diam(D)$, is equal to the greatest distance between any vertices in $D$. Namely
	$$Diam(G)=max\{dist(x,y):x,y \in V(D)\}$$
\end{definition}

\begin{definition}\cite{handbook}
	A digraph is called (strongly) connected if there exists a directed path between any pair of its vertices. 
\end{definition}

\begin{definition} \cite{handbook, digraph}
	An undirected graph $G$ is complete if for any $u,v \in V(G)$, $\{u,v\}$ is an edge. Similarly, a digraph $D$ is said to be complete if for any $u,v \in V(D)$, $u \to v$ is an arc.
\end{definition}

\begin{definition} \cite{digraph}
	Let $D$ be a digraph. If $v \to u$ is also an arc for every $u \to v$, then $D$ is called a symmetric digraph.
\end{definition}

\begin{remark}\cite{digraph}
	Essentially, any undirected graph can be considered as a symmetric digraph. 
\end{remark}

\begin{definition}\cite{rd} 
	Let $G=(V,E)$ and $H=(W,F)$ be two graphs. If $W \subseteq V$ and $F \subseteq E$, then $H$ is said to be a subgraph of $G$ ($G$ is said to be the supergraph of $H$) and this situation is denoted by $H \subseteq G$.
\end{definition}

\begin{definition} \cite{handbook}
	The trivial graph is the graph on one vertex without any edges.
\end{definition}

\begin{definition}\cite{digraph}
	Let $G=(V,E)$ and $H=(W,F)$ be two digraphs. $H$ is called a spanning subdigraph of $G$ if $H \subseteq G$ and $V=W$.
\end{definition}

\begin{remark}\label{bag}
	Let $G$ and $H$ be two digraphs such that $H$ is a spanning subdigraph of $G$. If $H$ is strongly connected, then so is $G$.
\end{remark}

\subsection{Preliminaries on zero-divisor graphs of semigroups and rings}
\begin{definition}\label{ilk}\cite{dfafsl}
	Let $R$ be a ring. The zero-divisor graph of $R$, denoted by $\overrightarrow{\Gamma}(R)$, is a simple digraph which has the set $Z(R)-\{0\}$ as its vertex set. For distinct $x,y \in Z(R)-\{0\}$,  $x \to y$ is an arc if and only if $xy =0$.
\end{definition}

Analogous to the definition of rings, the definition of the zero-divisor graph of a semigroup is as follows;

\begin{definition}\label{iki}\cite{wright}
	Let $S$ be a semigroup with $0$. The zero-divisor graph of $S$ denoted by $\overrightarrow{\Gamma}(S)$ is a simple digraph which has the set $Z(S)-\{0\}$ as its vertex set. For distinct $x,y \in Z(S)-\{0\}$, $x \to y$ is an arc if and only if $xy =0$. 
\end{definition}

\begin{remark}\cite{demeyer3}
	If $S$ is a commutative semigroup (or a commutative ring) with zero, then $x \to y$ being an arc in $\overrightarrow{\Gamma}(S)$ implies that $yx=0$ as well. Thus, the zero divisor graph of a commutative semigroup (or of a commutative ring) is undirected, and is denoted by $\Gamma(S)$.
\end{remark}

\begin{example}\cite{redmond}
	Let $S = \mathbb{Z}_{12}$. $Z(\mathbb{Z}_{12})-\{0\}=\{\overline{2}, \overline{3}, \overline{4},\overline{6},\overline{8},\overline{9},\overline{10}\}$ and Figure \ref{z12} demonstrates the zero-divisor graph of $\mathbb{Z}_{12}$. 
	
	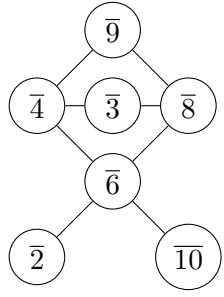
\begin{figure}[th]
		\centerline{\begin{tikzpicture}\label{z12}
				
				\node[circle, draw] (2) at (-1,-1) {$\overline{2}$};
				\node[circle, draw] (3) at (0,1) {$\overline{3}$};
				\node[circle, draw] (4) at (-1,1) {$\overline{4}$};
				\node[circle, draw] (6) at (0,0) {$\overline{6}$};
				\node[circle, draw] (8) at (1,1) {$\overline{8}$};
				\node[circle, draw] (9) at (0,2) {$\overline{9}$};
				\node[circle, draw] (10) at (1,-1) {$\overline{10}$};
				
				\draw (2) -- (6);
				\draw (3) -- (4);
				\draw (3) -- (8);
				\draw (4) -- (6);
				\draw (4) -- (9);
				\draw (6) -- (8);
				\draw (6) -- (10);
				\draw (8) -- (9);
				
		\end{tikzpicture} }
		\vspace*{8pt}
		\caption{Zero divisor graph of $\mathbb{Z}_{12}$}
		\label{z12}
	\end{figure}
	
	Figure \ref{z12y} demonstrates the zero-divisor graph of $\mathbb{Z}_{12}$ as a symmetric digraph as in the Definition \ref{ilk}.
	
	\begin{figure}[th]
		\centerline{\begin{tikzpicture}\label{z12y}
				
				\node[circle, draw] (2) at (-1,-1) {$\overline{2}$};
				\node[circle, draw] (3) at (0,1) {$\overline{3}$};
				\node[circle, draw] (4) at (-1,1) {$\overline{4}$};
				\node[circle, draw] (6) at (0,0) {$\overline{6}$};
				\node[circle, draw] (8) at (1,1) {$\overline{8}$};
				\node[circle, draw] (9) at (0,2) {$\overline{9}$};
				\node[circle, draw] (10) at (1,-1) {$\overline{10}$};
				
				\draw[->]  (2) -- (6);
				\draw[->]  (6) -- (2);
				\draw[->]  (3) -- (4);
				\draw[->]  (4) -- (3);
				\draw[->]  (3) -- (8);
				\draw[->]  (8) -- (3);
				\draw[->]  (4) -- (6);
				\draw[->]  (6) -- (4);
				\draw[->]  (4) -- (9);
				\draw[->]  (9) -- (4);
				\draw[->]  (6) -- (8);
				\draw[->]  (8) -- (6);
				\draw[->]  (6) -- (10);
				\draw[->]  (10) -- (6);
				\draw[->]  (8) -- (9);
				\draw[->]  (9) -- (8);		
		\end{tikzpicture}}
		\vspace*{8pt}
		\caption{Directed zero divisor graph of $\mathbb{Z}_{12}$}
		\label{z12y}
	\end{figure}
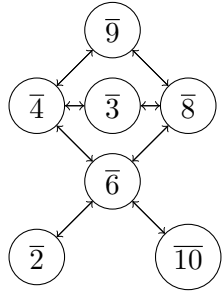
\end{example}

\begin{definition}\cite{extended}
	Let $R$ be a commutative ring. Then, the simple, undirected graph where $Z(R)-\{0\}$ is the vertex set and distinct $x,y$ are adjacent if and only if there exist $m,n \in \mathbb{Z}^+$ such that $x^ny^m=0$ and $x^n,y^m \neq 0$ is called the extended zero-divisor graph of $R$.
\end{definition}

\begin{example}
	In $\mathbb{Z}_{12}$, for any $n \in \mathbb{Z}^+$,
	$$
	\begin{array}{ll}
		\overline{2}^n \in \{\overline{2},\overline{4},8\} & \quad \overline{3}^n \in \{\overline{3},\overline{9}\} \\
		\overline{4}^n = \overline{4}, & \quad \overline{6}^n \in \{\overline{0},\overline{6}\} \\
		\overline{8}^n \in \{\overline{4},\overline{8}\} & \quad \overline{9}^n = \overline{9}, \\
		\multicolumn{2}{c}{\overline{10}^n \in \{\overline{4},\overline{10}\}.} 
	\end{array}
	$$ 
	Since $\overline{10}.\overline{3}=\overline{6}\neq\overline{0}$, $\overline{10}$ and $\overline{3}$ are not adjacent in the zero-divisor graph. However $\overline{10}^2.\overline{3}=\overline{4}.\overline{3}=\overline{0}$. Thus, $\overline{10}$ and $\overline{3}$ are adjacent in the extended zero-divisor graph.
	Figure \ref{z12e} demonstrates the extended zero-divisor graph of $\mathbb{Z}_{12}$. The edges of the extended zero-divisor graph that are not edges of the zero-divisor graph are dashed.
	
	\begin{figure}[th]
		\centerline{\begin{tikzpicture}\label{z12e}
				
				\node[circle, draw] (2) at (-1,-1) {$\overline{2}$};
				\node[circle, draw] (3) at (0,1) {$\overline{3}$};
				\node[circle, draw] (4) at (-1,1) {$\overline{4}$};
				\node[circle, draw] (6) at (0,0) {$\overline{6}$};
				\node[circle, draw] (8) at (1,1) {$\overline{8}$};
				\node[circle, draw] (9) at (0,2) {$\overline{9}$};
				\node[circle, draw] (10) at (1,-1) {$\overline{10}$};
				
				\draw (2) -- (6);
				\draw[dashed] (2) -- (3);
				\draw[dashed] (2) to[bend left=6] (9);
				
				\draw (3) -- (4);
				\draw (3) -- (8);
				
				\draw (4) -- (6);
				\draw (4) -- (9);
				
				\draw (6) -- (8);
				\draw (6) -- (10);
				
				\draw (8) -- (9);
				
				\draw[dashed] (10) to[bend right=6] (9);
				\draw[dashed] (10) -- (3);

		\end{tikzpicture}}
		\vspace*{8pt}
		\caption{Extended zero divisor graph of $\mathbb{Z}_{12}$}
		\label{z12e}
	\end{figure}
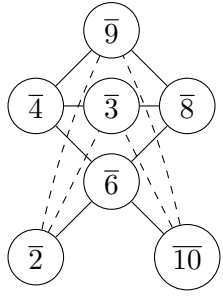
\end{example}

\section{Results}\label{sec3}

\begin{definition}
	Let $S$ be a non-commutative semigroup with $0$. The extended zero divisor of $S$, denoted $\overrightarrow{\Gamma}_E (S)$, is a graph with $Z(S)-\{0\}$ as the set of vertices. $x \to y$ is an edge of $\overrightarrow{\Gamma}_E (S)$ if and only if there exist positive integers $n$ and $m$ such that $x^n y^m = 0$ with both $x^n \neq 0$ and $y^m \neq 0$.
\end{definition}

\begin{proposition}
	Let $S$ be a non-commutative semigroup with $0$. $\overrightarrow{\Gamma}(S)$ is a spanning subdigraph of 	$\overrightarrow{\Gamma}_E (S)$.
\end{proposition}
\begin{proof}
	The vertex set of both graphs is $Z(S)-\{0\}$. Let now $x,y \in Z(S)-\{0\}$ and $x \to y$ be an edge of $\overrightarrow{\Gamma}(S)$. Then $xy=0$. Thus, $x \to y$ is also an edge of $\overrightarrow{\Gamma}_E(S)$.
\end{proof}

\begin{lemma}\label{tumevar}
	Let $S$ be a non-commutative semigroup with $0$, $x \in S - \{0\}$, and $n \in \mathbb{Z}^{+}$. If $x \notin Nil(S)$, then 
	
	\begin{align}
		Ann_l(x^2) = Ann_l(x) &\iff Ann_l(x^n) = Ann_l(x), \quad \forall n \geq 2, \label{sol}\\
		Ann_r(x^2) = Ann_r(x) &\iff Ann_r(x^n)= Ann_r(x), \quad \forall n \geq 2. \label{sag} 
	\end{align}
\end{lemma}

\begin{proof} 
	Let $Ann_l(x^2) = Ann_l(x)$. The proof of Eq. (\ref{sol}) will be carried out by induction on $n$.  
	\begin{align*}
		y \in Ann_l (x^3) &\Rightarrow yx^3=0 \\
		&\Rightarrow yx \in Ann_l(x^2) = Ann_l(x) \\
		&\Rightarrow (yx)x=yx^2=0 \\
		&\Rightarrow y \in Ann_l(x^2) = Ann_l(x) \\
		&\Rightarrow yx = 0 \\
		&\Rightarrow y \in Ann_l(x).
	\end{align*}
	Hence, $Ann_l(x^3)=Ann_l(x)$. Assume next $Ann_l(x^k) = Ann_l(x)$ for a $k \geq 1$ and $y \in Ann_l(x^{k+1})$.
	\begin{align*}
		y \in Ann_l (x^{k+1}) &\Rightarrow yx^{k+1}=0 \\
		&\Rightarrow yx \in Ann_l(x^k) = Ann_l(x) \\
		&\Rightarrow (yx)x=yx^2=0 \\
		&\Rightarrow y \in Ann_l(x^2) = Ann_l(x) \\
		&\Rightarrow yx = 0 \\
		&\Rightarrow y \in Ann_l(x).
	\end{align*}
\end{proof}
Therefore, $Ann_l(x^n) = Ann_l(x)$ for all $n \in \mathbb{Z}^+$. The converse is straightforward since $Ann_l(x^n) = Ann_l(x)$ for all $n \in \mathbb{Z}^+$ and as $2 \in \mathbb{Z}^+$. 

A similar approach can be used to prove Eq. (\ref{sag}).
\begin{remark}
	Let $S$ be a non-commutative semigroup with $0$ and $x$ be a nilpotent element in $S$ such that $n_x \geq 3$. Then $n_x - 1 \geq 2$. If $Ann(x^{n_x - 1})=Ann(x)$, then $x \in Ann(x)$, hence $x^2 = 0$. This is a contradiction since $n_x \geq 3$. Consequently, $Ann(x) \neq Ann(x^{n_x - 1})$. 
\end{remark}

\begin{theorem}\label{g}
	Let $S$ be a non-commutative semigroup with $0$. If $\overrightarrow{\Gamma}_E (S) = \overrightarrow{\Gamma}(S)$ then,
	\begin{enumerate}
		\item $Ann(y^2) = Ann(y)$ for all non-zero $y \in Z(S)-Nil(S)$. \\
		\item  If $Nil(S) \neq \{0\}$, for all $x \in Nil(S)-\{0\}$, $n_x \leq 3$ and  if $n_x = 3$ then $Ann(x^{n_x-1})-Ann(x) = \{x\}$.
	\end{enumerate}
\end{theorem}
\begin{proof}
	\begin{enumerate}
		\item Let $0 \neq y \in Z(S)-Nil(S)$ and $y \neq s \in Z(S)-\{0\}$.
		
		\begin{align*}
			s \in Ann(y^2) &\Rightarrow sy^2 = 0 = y^2s \\
			&\Rightarrow s \to y \text{ and } y \to s \text{ are arcs of } \overrightarrow{\Gamma}_E (S) \\
			&\Rightarrow s \to y \text{ and } y \to s \text{ are arcs of } \overrightarrow{\Gamma} (S) \\
			&\Rightarrow sy = 0 = ys \\
			&\Rightarrow s \in Ann(y).
		\end{align*}
		Thus, $ Ann(y^2) \subseteq Ann(y)$. Therefore, $Ann(y^2) = Ann(y)$ since $Ann(y) \subseteq Ann(y^2)$. 
		
		\item If $n_x \geq 4$, then $x \to x^2$ is an arc of $\overrightarrow{\Gamma}_E (S)$ but not an arc of $\overrightarrow{\Gamma}(S)$ since $x^{n_x -2} x^2=0$ while both $x^2$ and $x^{n_x -2}$ are non-zero whereas on the contrary, $xx^2 = x^3 \neq 0$. This is a contradiction. Thus, $n_x \leq 3$. Let $n_x = 3$. Assume that $Ann(x^{n_x-1})-Ann(x) \neq \{x\}$. Then, $z \in Ann(x^{n_x -1})-Ann(x)$. Therefore, $zx^{n_x -1} = 0 = x^{n_x - 1}z$. Which further implies that both $x \to z$ and $z \to x$ are edges of $\overrightarrow{\Gamma}_E (S)$, whereas $z \notin Ann(x)$ implies that either $zx \neq 0$ or $xz \neq 0$, from which it follows that either $z \to x$ or $x \to z$ is not an edge of $\overrightarrow{\Gamma}(S)$. This is a contradiction given $\overrightarrow{\Gamma}_E (S) = \overrightarrow{\Gamma}(S)$. Thus, $Ann(x^{n_x-1})-Ann(x) = \{x\}$. 
	\end{enumerate}
\end{proof} 

\begin{theorem}\label{gg}
	The following are equivalent for a non-commutative semigroup $S$ with $0$.
	\begin{enumerate}
		\item $\overrightarrow{\Gamma}_E (S) = \overrightarrow{\Gamma}(S)$ 
		\item If $Nil(S) \neq \{0\}$, then for all $x \in Nil(S)$, $n_x \leq 3$ and if $n_x = 3$ then $Ann_{l}(x^{2})-Ann_{l}(x) = \{x\}$ and $Ann_{r}(x^{2})-Ann_{r}(x) = \{x\}$. Moreover $Ann_l(y^2) = Ann_l(y)$ and $Ann_r(y^2) = Ann_r(y)$ for all non-zero $y \in Z(S)-Nil(S)$. 
		\item If $Nil(S) \neq \{0\}$, then for all $x \in Nil(S)$, $n_x = 2$ or $n_x = 3$ only when $Ann_{l}(x^{2})-Ann_{l}(x) = \{x\}$ and $Ann_{r}(x^{2})-Ann_{r}(x) = \{x\}$. Moreover $\sqrt{Ann_{l}(y)}-Nil(S) \subseteq Ann_{l}(y)$ and $\sqrt{Ann_{r}(y)}-Nil(S) \subseteq Ann_{r}(y)$ for all  $y \in Z(S)-\{0\}$.
	\end{enumerate}
\end{theorem}

\begin{proof}
	$(1) \Rightarrow (2):$ 	
	Let $x \in Nil(S)$ and $n_x = 3$. Assume that $Ann_r(x^{n_x-1})-Ann_r(x) \neq \{x\}$. Then, there exist a $x \neq z \in Z(S) - \{0\}$ such that $z \in Ann_r(x^{n_x -1})-Ann_r(x)$. Therefore, $x^{n_x - 1}z = 0$. Which further implies that both $x \to z$ is an arc of $\overrightarrow{\Gamma}_E (S)$ whereas $z \notin Ann_r(x)$ implies that $xz \neq 0$ from which it follows $x \to z$ is not an edge of $\overrightarrow{\Gamma}(S)$. This is a contradiction given $\overrightarrow{\Gamma}_E (S) = \overrightarrow{\Gamma}(S)$. Thus, $Ann_r(x^{n_x-1})-Ann_r(x) = \{x\}$. $Ann_l(x^{n_x-1})-Ann_(x) = \{x\}$ follows similarly.
	
	Let $\overrightarrow{\Gamma}_E (S) = \overrightarrow{\Gamma}(S)$, $y \in Z(S) \setminus Nil(S)$, and $s \in S \setminus \{0\}$. Since $Ann_r(y) \subseteq Ann_r(y^2)$ and $Ann_l(y) \subseteq Ann_l(y^2)$, it is sufficient to show that $Ann_r(y^2) \subseteq Ann_r(y)$ and $Ann_l(y^2) \subseteq Ann_l(y)$.
	\begin{align*}
		s \in Ann_r(y^2) &\Rightarrow y^2s = 0 \\
		&\Rightarrow y \to s \text{ is an arc in the directed graph } \overrightarrow{\Gamma}_E(S) \\
		&\Rightarrow y \to s \text{ is an arc in the directed graph } \overrightarrow{\Gamma}(S) \\
		&\Rightarrow ys = 0 \\
		&\Rightarrow s \in Ann_r(y).
	\end{align*}
	Hence, $Ann_r(y) = Ann_r(y^2)$. Similarly, it can be shown that $Ann_l(y) = Ann_l(y^2)$.
	
	$(2) \Rightarrow (3):$ Let $z,y \in Z(S)-\{0\}$. 
	\begin{align*}
		z \in \sqrt{Ann_{l}(y)}-Nil(S) &\Rightarrow z^n \in Ann_{l}(y), \quad \exists n \in \mathbb{Z}^{+} \\
		&\Rightarrow z^n y = 0 \\
		&\Rightarrow y \in Ann_{r}(z^n) = Ann_r(z), \quad \text{ by Eq. } (\ref{sag}) \\
		&\Rightarrow zy = 0 \\
		&\Rightarrow z \in Ann_{l}(y).
	\end{align*}
	
	Using the same method, one can see that $\sqrt{Ann_{r}(y)}-Nil(S) \subseteq Ann_{r}(y)$ holds.
	
	$(3) \Rightarrow (1):$ Let $x \to y$ be an arc of $\overrightarrow{\Gamma}_E (S)$. Then, there exist positive integers $n,m$ such that $x^n y^m = 0$ where $x^n, y^m \neq 0$. Consider the following cases.
	
	\begin{enumerate}
		\item Let $x,y \in Nil(S)-\{0\}$. 
		\begin{enumerate}
			\item Let $n_x = n_y = 2$, then $n=m=1$ and $xy=0$. Therefore $x \to y$ is an arc of $\overrightarrow{\Gamma}(S)$. \\
			\item Let $n_x = n_y = 3$. 
			
			If $n=m=1$, then $xy=0$. Thus, $x \to y$ is an arc of $\overrightarrow{\Gamma}(S)$. 
			
			If $n=2$, $m=1$, then 
			\begin{align*}
				x^2y=0 &\Rightarrow y \in Ann_{r}(x^2)=Ann_{r}(x) \cup \{x\} \\
				&\Rightarrow y \in Ann_{r}(x) \\
				&\Rightarrow xy = 0 \\
				&\Rightarrow x \to y \text{ is an arc of } \overrightarrow{\Gamma}(S).
			\end{align*}
			
			The instance where $n=1$ and $m=2$ follows similarly.
			
			If $n=m=2$, then 
			\begin{align*}
				x^2y^2=0 &\Rightarrow y^2 \in Ann_{r}(x^2)=Ann_{r}(x) \cup \{x\} \\
				&\Rightarrow y^2 \in Ann_{r}(x) \\
				&\Rightarrow xy^2 = 0 \\
				&\Rightarrow x \in Ann_{l}(y^2) = Ann_{l}(y) \cup \{y\} \\
				&\Rightarrow x \in Ann_{l}(y) \\
				&\Rightarrow xy = 0 \\
				&\Rightarrow x \to y \text{ is an arc of } \overrightarrow{\Gamma}(S).
			\end{align*} \label{bnz}
			\item Let either $n_x= 2$ and $n_y=3$ or $n_x = 3$ and $n_y=2$. This follows similarly to \ref{bnz}. Thus, $x \to y$ is an arc of $\overrightarrow{\Gamma}(S)$.
			
		\end{enumerate}
		
		\item Let $x \in Nil(S)-\{0\}$ and $y \in Z(S)-Nil(S)$.
		\begin{enumerate}
			\item If  $n_x=2$, then \label{bnza}
			\begin{align*}
				xy^m = 0 &\Rightarrow y^m \in Ann_{r}(x) \\
				&\Rightarrow y \in \sqrt{Ann_{r}(x)} - Nil(S) \subseteq {Ann_r(x)} \\
				&\Rightarrow  y \in {Ann_r(x)} \\
				&\Rightarrow xy = 0 \\
				&\Rightarrow x \to y \text{ is an arc of } \overrightarrow{\Gamma}(S).
			\end{align*}
			
			\item If $n_x = 3$, then either $n=1$ or $n=2$. If $n=1$, then $xy^m=0$ and this case follows similarly to \ref{bnza}. If $n=2$, then
			
			\begin{align*}
				x^2y^m = 0 &\Rightarrow y^m \in Ann_{r}(x^2) = Ann_{r}(x) \cup \{x\} \\
				&\Rightarrow y^m \in Ann_{r}(x) \\
				&\Rightarrow y \in \sqrt{Ann_{r}(x)} - Nil(S) \subseteq {Ann_r(x)} \\
				&\Rightarrow  y \in {Ann_r(x)} \\
				&\Rightarrow xy = 0 \\
				&\Rightarrow x \to y \text{ is an arc of } \overrightarrow{\Gamma}(S).
			\end{align*} 
		\end{enumerate}
		Let $x \in Z(S)-Nil(S)$ and $y \in Nil(S)-\{0\}$. This case can be proved analogously to \ref{bnz2}.
		\label{bnz2}
		\item Let $x,y \in Z(S)-Nil(S)$.
		\begin{align*}
			x^ny^m=0 &\Rightarrow x^n \in Ann_{l}(y^m) \\
			&\Rightarrow x \in \sqrt{Ann_{l}(y^m)} - Nil(S) \subseteq Ann_{l}(y^m) \\
			&\Rightarrow x \in Ann_l(y^m) \\
			&\Rightarrow xy^m = 0 \\
			&\Rightarrow y^m \in Ann_{r}(x) \\
			&\Rightarrow y \in \sqrt{Ann_{r}(x)}-Nil(S) \subseteq  Ann_{r}(x) \\
			&\Rightarrow xy = 0 \\
			&\Rightarrow x \to y \text{ is an arc of } \overrightarrow{\Gamma}(S).
		\end{align*}
	\end{enumerate}
\end{proof}

\begin{corollary}\label{ggg}
	Let $S$ be a non-commutative semigroup with $0$. $\overrightarrow{\Gamma}_E(S) \neq \overrightarrow{\Gamma}(S)$ if there exist an $x \in Nil(S)$ such that $n_x \geq 4$.
\end{corollary}

\begin{proposition}
	Let $S$ be a non-commutative semigroup with $0$. If $Nil(S)=\{0\}$, then $\overrightarrow{\Gamma}_E (S) = \overrightarrow{\Gamma}(S)$.
\end{proposition}

\begin{proof} 
	Assume that $Nil(S)=\{0\}$ and that there exist an element $x \in Z(S)-\{0\}$ such that $Ann(x)\neq Ann(x^2)$. Then, by Theorem \ref{g}, $\overrightarrow{\Gamma}_E (S) \neq \overrightarrow{\Gamma}(S)$. Since $Ann(x) \subseteq Ann(x^2)$, there exist $z \in Z(S)-\{0\}$ such that $z \notin Ann(x)$ whereas $z \in Ann(x^2)$. This further implies that either $zx \neq 0$ or $xz\neq 0$ while $zx^2=0$ and $x^2z=0$. Consider following two cases:
	
	\begin{enumerate}
		\item Let $zx = 0$ and $xz \neq 0$. Then
		\begin{align}\label{33}
			(xz)^2=(xz)(xz)=x(zx)z=0.
		\end{align}
		Eq. \ref{33} implies that $0 \neq xz \in Nil(S)$, which is a contradiction. One can prove the case where $zx \neq 0$ and $xz = 0$, similarly.
		\item Let $zx \neq 0$ and $xz \neq 0$. Consider $xzx \in Z(S)$.
		\begin{enumerate}
			\item If $xzx = 0$, then 
			\begin{align}
				(xz)^{2}=(xz)(xz)=(xzx)z=0.
			\end{align}
			This is a contradiction since $Nil(S)=\{0\}$ whereas the equality implies $0 \neq xz \in Nil(S)-\{0\}$.
			\item If $xzx \neq 0$, then
			\begin{align}
				(xzx)^{2}=(xzx)(xzx)=x(zx^2)(zx)=x0(zx)=0.
			\end{align}
			This yields a contradiction, because $Nil(S)=\{0\}$ while the equality implies that $xzx \in Nil(S)$ with $xzx \neq 0$.
		\end{enumerate}
	\end{enumerate}
	Since every possible case contradicts with $Nil(S)=\{0\}$, then $\overrightarrow{\Gamma}_E (S) = \overrightarrow{\Gamma}(S)$.
\end{proof}

This implies that $\overrightarrow{\Gamma}(S)$ is connected if $Z(S)=Nil(S)$. Since $\overrightarrow{\Gamma}(S)$ is a spanning subgraph of $\overrightarrow{\Gamma}_E(S)$, if $\overrightarrow{\Gamma}(S)$ is connected, by Remark \ref{bag}, so is $\overrightarrow{\Gamma}_E(S)$.
\begin{theorem} \label{nx}
	Let $S$ be a non-commutative semigroup with $0$. There exist an element $x \in Nil(S)$ such that for all $y \in Z(S)-\{0\}$, $x \to y$ and $y \to x$ are arcs of $\overrightarrow{\Gamma}_E (S)$ if and only if $Z(S)-\{0\}=\sqrt{Ann(x^{n_x - 1})-\{0\}}$.
\end{theorem}
\begin{proof}
	Let $x \in Nil(S)$ and $x \to y$ and $y \to x$ be arcs of $\overrightarrow{\Gamma}_E (S)$. Then, there exist $a,b,m,n \in \mathbb{Z}^{+}$ such that $x^n y^m =0$ and $y^a x^b=0$ where $x^n,y^m,x^b,y^a \neq 0$. Since $x$ is nilpotent, we have $n \leq n_x - 1$ and $b \leq n_x - 1$. Thus, 
	\begin{align}
		x^{n_x - 1}y^m = 0 = y^ax^{n_x -1}.
	\end{align}
	Without the loss of generality, assume that $a \leq m$. This furthermore implies that
	\begin{align}
		x^{n_x - 1}y^m = 0 = y^m x^{n_x -1}.
	\end{align}
	Then, $y^m \in Ann(x^{n_x -1})-\{0\}$. Moreover $y \in \sqrt{ Ann(x^{n_x -1})-\{0\}}$. 
	
	Thus, $Z(S)-\{0\}=\sqrt{ Ann(x^{n_x -1})-\{0\}}$.
	
	Consider the converse of the proof. Namely, let $Z(S)-\{0\}=\sqrt{Ann(x^{n_x - 1})-\{0\}}$. Hence, for all $y \in Z(S)-\{0\}$, there exist an $n \in \mathbb{Z}^+$ such that $y^n \in Ann(x^{n_x - 1})-\{0\}$. Then, $y^nx^{n_x - 1} = 0 = x^{n_x - 1}y^n$. Thus, $x \to y$ and $y \to x$ are arcs of $\overrightarrow{\Gamma}_E (S)$. Thus, $x$ is adjacent to and adjacent from all $y \in Z(S)-\{0\}$. \end{proof}

\begin{proposition}\label{idempotent}
	Let $S$ be a non-commutative semigroup with $0$. If $x \in Z(S)-Nil(S)$ and $x$ is adjacent to or adjacent from $y$ in $\overrightarrow{\Gamma}_E (S)$, for all $y \in Z(S)-\{0\}$, then $x$ is idempotent. 
\end{proposition}
\begin{proof}
	Let $x \in Z(S)-Nil(S)$ such that for all $y \in Z(S)-\{0\}$ at least one of $x \to y$ or $y \to x$ is an arc of $\overrightarrow{\Gamma}_E (S)$. 
	
	Let $x^2 \neq x$. Then, $x$ is adjacent to $x^2$ (or $x^2$ is adjacent to $x$), implying that there exist $a,b \in \mathbb{Z}^{+}$ such that $x^a ({x^2})^b = x^{a+2b} = 0$ (or $({x^2})^b x^a = 0$). This is a contradiction since $x \notin Nil(S)$.
\end{proof}

\begin{corollary}\label{idemp}
	Let $S$ be a non-commutative semigroup with $0$. If	$\overrightarrow{\Gamma}_E (S)$ is a complete graph, then $Z(S)-Nil(S)=Z(S)\cap E(S)$. Namely, for any $x \in Z(S)-\{0\}$, under this condition, $x$ is either nilpotent or idempotent.
\end{corollary} 
\begin{proof}
	Straightforward from Theorem \ref{nx} and Proposition \ref{idempotent}.
\end{proof}

\begin{corollary}\label{id}
	Let $x \in Z(S)-\{0\}$ and, for all $y \in Z(S)-\{0\}$, let $x \to y$ and $y \to x$ be arcs of $\overrightarrow{\Gamma}_E (S)$ for a non-commutative semigroup $S$ with zero. Then, there exist $n \in \mathbb{Z}^+ $ such that $Z(S)-\{0\} =\sqrt{Ann(x^n)-\{0\} }$.
\end{corollary}
\begin{proof}
	If $x \in Nil(S)$, then by Theorem \ref{nx} it follows that $Z(S)-\{0\} =\sqrt{Ann(x^{n_x - 1})-\{0\} }$. 
	
	Let $x \notin Nil(S)$. Then, by Proposition \ref{idempotent}, $x$ is an idempotent. Thus, for all integers $n \geq 2$, $x^n = x$. Hence $\sqrt{Ann(x^n)}=\sqrt{Ann(x)}$. 
	
	Consider $y \in Z(S)-\{0\}$. By hypothesis, $x \to y$ and $y \to x$ are arcs of $\overrightarrow{\Gamma}_E (S)$. Then, there exist positive integers $a,b$ such that 
	\begin{align*}
		x y^a = 0 = y^b x
	\end{align*}
	with $y^a, y^b \neq 0$. Then, either $a \leq b$ or $b \leq a$. Assume the latter. In this case, we have
	\begin{align}
		x y^a = 0 = y^a x,
	\end{align}
	i.e. $y^a \in Ann(x) -\{0\}$. Therefore, $y \in \sqrt{Ann(x)-\{0\}}$. Thus, $Z(S)-\{0\} \subseteq \sqrt{Ann(x)-\{0\}}$. Furthermore, since $\sqrt{Ann(x)-\{0\}} \subseteq Z(S)-\{0\}$,
	$Z(S)-\{0\}=\sqrt{Ann(x)-\{0\}}$ is obtained.
\end{proof}

\begin{theorem} 
	Let $S$ be a non-commutative semigroup with $0$. $\overrightarrow{\Gamma}_E (S)$ is a complete digraph if and only if one of the following holds for any $x, y \in Z(S)-\{0\}$, 
	
	\begin{enumerate}
		\item If both $x,y \notin Nil(S)$, then $x \to y$ and $y \to x$ are arcs of $\overrightarrow{\Gamma}(S)$. \\
		\item If $x \in Nil(S)$ while $y \notin Nil(S)$, then $x \in \sqrt{Ann(y)-\{0\}}$ (or, if $y \in Nil(S)$ while $x \notin Nil(S)$, then $y \in \sqrt{Ann(x)-\{0\}}$). \\
		\item If both $x,y \in Nil(S)$, then $x \in \sqrt{Ann(y^{n_y -1})-\{0\} }$ and $y \in \sqrt{Ann(x^{n_x -1})-\{0\} }$.
	\end{enumerate}
\end{theorem}
\begin{proof}
	\textbf{Only if part.} Since $\overrightarrow{\Gamma}_E (S)$ is a complete digraph, for all $x \in Z(S)-\{0\}$, $x$ is either nilpotent or idempotent by Corollary \ref{idemp}, and there exist $n,m,\alpha,\beta \in \mathbb{Z}^+$ such that $x^ny^m=0 = y^{\alpha} x^{\beta}$ where  $x^n \neq 0, y^m \neq 0, y^{\alpha} \neq 0, x^{\beta} \neq 0$.
	\begin{enumerate}
		\item Let $x,y \notin Nil(S)$. Then $x,y \in E(S)$. Thus, 
		$$\begin{aligned} 
			x^ny^m=0 = y^{\alpha} x^{\beta} &\Rightarrow xy = 0 = yx \\
			&\Rightarrow x \to y \text{ and } y \to x \text{ are arcs of } \overrightarrow{\Gamma}(S).
		\end{aligned}$$
		
		\item Let $x \in Nil(S)$ and $y \notin Nil(S)$. This implies that $y \in E(S)$. Hence, 
		$$\begin{aligned}
			x^ny^m=0=y^{\alpha}x^{\beta} &\Rightarrow x^ny=0=yx^{\beta}\\
			&\Rightarrow x^{n_x -1}y=0=yx^{n_x -1}\\
			&\Rightarrow  x^{n_x -1} \in Ann(y)-\{0\} \\
			&\Rightarrow  x \in \sqrt{Ann(y)-\{0\} }.
		\end{aligned}$$
		
		\item Let $x,y \in Nil(S)$. 
		$$\begin{aligned}
			x^ny^m=0=y^{\alpha}x^{\beta} &\Rightarrow x^{n_x -1}y^{n_y -1}=0=y^{n_y -1}x^{n_x -1} \\
			&\Rightarrow  x^{n_x -1} \in Ann(y^{n_y -1})-\{0\} \text{ and } y^{n_y -1} \in Ann(x^{n_x -1})-\{0\} \\
			&\Rightarrow x \in \sqrt{Ann(y^{n_y -1})-\{0\} } \text{ and } y \in \sqrt{Ann(x^{n_x -1})-\{0\} }. 
		\end{aligned}$$
	\end{enumerate}
	
	\textbf{The if part}. Let $x,y \in Z(S)-\{0\}$.
	
	\begin{enumerate}
		\item If $x,y \notin Nil(S)$, then, by 1, $x \to y$ and $y \to x$ are arcs of $\overrightarrow{\Gamma}(S)$. Thus, $x\to y$ and $y \to x$ are arcs of $\overrightarrow{\Gamma}_E(S)$.
		
		\item If $x \in Nil(S)$ and $y \notin Nil(S)$, then
		\begin{align*}
			x \in \sqrt{Ann(y)-\{0\}} &\Rightarrow x^n \in Ann(y)-\{0\}, \quad \exists n \in \mathbb{Z}^+ \\
			&\Rightarrow x^n y = 0 = y x^n \\
			&\Rightarrow x \to y \text{ and } y \to x \text{ are arcs of } \overrightarrow{\Gamma}_E (S).
		\end{align*}
		\item If $x,y \in Nil(S)$, then, by 3., 
		\begin{align*}
			x \in \sqrt{Ann(y^{n_y -1})-\{0\}}  &\Rightarrow x^n \in Ann(y^{n_y -1})-\{0\} \quad \exists n \in \mathbb{Z}^+ \\
			&\Rightarrow x^ny^{n_y -1} = 0 = y^{n_y -1}x^n \\
			&\Rightarrow x \to y \text{ and } y \to x \text{ are arcs of }\overrightarrow{\Gamma}_E(S).
		\end{align*}
		and 
		\begin{align*}
			y \in \sqrt{Ann(x^{n_x -1})-\{0\}} &\Rightarrow y^m \in Ann(x^{n_x -1})-\{0\} \quad \exists m \in \mathbb{Z}^+ \\
			&\Rightarrow y^m x^{n_x -1} = x^{n_x -1} y^m \\ 
			&\Rightarrow x \to y \text{ and } y \to x \text{ are arcs of }\overrightarrow{\Gamma}_E(S).
		\end{align*}
	\end{enumerate}
	Thus, any pair of distinct vertices are connected by two symmetric arcs. Hence, $\overrightarrow{\Gamma}_E(S)$ is a complete digraph.
\end{proof}

Let $\overline{Z}(S)=\{x^{n_x -1} : x \in Nil(S) -\{0\}\}$ for semigroup $S$ with $0$ \cite{extended}.

\begin{corollary}
	Let $\overrightarrow{\Gamma}_E (S) \neq \overrightarrow{\Gamma}(S)$ for a non-commutative semigroup $S$ with $0$. Then,	$\overrightarrow{\Gamma}_E (S)$ is a complete digraph if $Z(S)=Nil(S)$ and $\overline{Z}(S)^{2}=\{0\}$. 
\end{corollary}
\begin{proof}
	Let $x,y \in Z(S)-\{0\}$. Since $Z(S) = Nil(S)$, $x,y \in Nil(S)$ with $x^{n_x - 1}, y^{n_y - 1} \neq 0$ and $\overline{Z}(S)^{2}=\{0\}$ implies that $x^{n_x - 1} y^{n_y - 1} = 0$. Thus, $x \to y$ is an arc of 	$\overrightarrow{\Gamma}_E (S)$. 
\end{proof}

In \cite{wright}, Wright proved that $\overrightarrow{\Gamma}(S)$ is strongly connected if and only if all zero divisors of $S$ are two sided. Moreover, the author proved that the diameter of $\overrightarrow{\Gamma}(S)$ is at most 3 if it is connected. 

\begin{theorem}
	Let $Z(S)=Nil(S)\neq\{0\}$.
	\begin{enumerate}
		\item If $|Z(S)\setminus\{0\}|=1$, then $\overrightarrow{\Gamma}_E(S)$ is the trivial graph and $Diam(\overrightarrow{\Gamma}_E(S))=0$.
		
		\item If $|Z(S)\setminus\{0\}| \geq 2$, then $\overrightarrow{\Gamma}_E(S)$ is strongly connected and $
		Diam(\overrightarrow{\Gamma}_E(S))\leq 2$.
		Moreover,
		\begin{enumerate}
			\item
			$
			Diam(\overrightarrow{\Gamma}_E(S))=1
			\Longleftrightarrow
			\overline{Z}(S)^2=\{0\},
			$
			
			\item
			$
			Diam(\overrightarrow{\Gamma}_E(S))=2
			\Longleftrightarrow
			\overline{Z}(S)^2\neq\{0\}.
			$
		\end{enumerate}
	\end{enumerate}
\end{theorem}

\begin{proof}
	\begin{enumerate}
		\item Follows immediately from the definition of the trivial graph.
		
		\item Let $x,y\in Z(S)\setminus\{0\}$ with $x\neq y$, and suppose that there is no directed edge $x\to y$ in $\overrightarrow{\Gamma}_E(S)$. Then $x^{n_x-1}y^{n_y-1}\neq0$
		Moreover,
		$ x^{n_x-1}y^{n_y-1}\notin\{x,y\}$,
		since
		\begin{align}
			x=x^{n_x-1}y^{n_y-1}
			\Rightarrow
			xy=x^{n_x-1}y^{n_y}=0
		\end{align}
		which implies that $x\to y$ is an edge of $\overrightarrow{\Gamma}_E(S)$, a contradiction.
		
		Similarly,
		\begin{align}
			y=x^{n_x-1}y^{n_y-1}
			\Rightarrow
			xy=x^{n_x}y^{n_y-1}=0
		\end{align}
		which again implies that $x\to y$ is an edge of $\overrightarrow{\Gamma}_E(S)$, another contradiction.
		
		Hence,
		\begin{align}
			x \to x^{n_x-1}y^{n_y-1} \to y
		\end{align}
		is a path in $\overrightarrow{\Gamma}_E(S)$,
		and therefore every pair of nonadjacent vertices has a directed path of length at most $2$ between them. Consequently,
		\begin{align}
			Diam(\overrightarrow{\Gamma}_E(S))\le 2
		\end{align}
		and $\overrightarrow{\Gamma}_E(S)$ is strongly connected.
		
		\begin{enumerate}
			\item We have
			\begin{align*}
				Diam(\overrightarrow{\Gamma}_E(S))=1
				&\Longleftrightarrow
				x\to y \text{ is an edge of } \overrightarrow{\Gamma}_E(S)
				\text{ for all }x,y\in Z(S)\setminus\{0\} \\
				&\Longleftrightarrow
				x^{n_x-1}y^{n_y-1}=0
				\text{ for all }x,y\in Z(S)\setminus\{0\} \\
				&\Longleftrightarrow
				\overline{Z}(S)^2=\{0\}
			\end{align*}
			
			\item Suppose that
			$Diam(\overrightarrow{\Gamma}_E(S))=2$.
			Then there exist distinct vertices $x,y\in Z(S)\setminus\{0\}$ such that $x\to y$ is not an edge of $\overrightarrow{\Gamma}_E(S)$. Hence,
			$x^{n_x-1}y^{n_y-1}\neq0$,
			which implies that
			$\overline{Z}(S)^2\neq\{0\}$.
			
			Conversely, if
			$\overline{Z}(S)^2\neq\{0\}$,
			then there exist $x,y\in Z(S)\setminus\{0\}$ such that
			$x^{n_x-1}y^{n_y-1}\neq0$.
			Therefore, $x\to y$ is not an edge of $\overrightarrow{\Gamma}_E(S)$. Since
			$Diam(\overrightarrow{\Gamma}_E(S))\le2$,
			it follows that
			$Diam(\overrightarrow{\Gamma}_E(S))=2$.
		\end{enumerate}
	\end{enumerate}
\end{proof}

\begin{theorem}
	Let $S$ be a non-commutative semigroup with $0$ and $\overrightarrow{\Gamma}_E(S) \neq \overrightarrow{\Gamma}(S)$. Then, $\overrightarrow{\Gamma}_E(S)$ contains a cycle. 
\end{theorem}

\begin{proof}	
	Let $\overrightarrow{\Gamma}_E(S) \neq \overrightarrow{\Gamma}(S)$. Then, by Theorem \ref{gg}, one of the followings hold:
	
	\begin{enumerate}
		\item There exist an $x \in Nil(S)$ such that $n_x \geq 4$,
		\item There exist an $x \in Nil(S)$ such that $n_x = 3$ where $Ann_r(x^2) \setminus Ann_r (x) \neq \{x\}$ or $Ann_l(x^2) \setminus Ann_l(x) \neq \{x\}$.
	\end{enumerate}
	
	Let $x \in Nil(S)$ such that $n_x \geq 3$. Then, $x x^{n_x -1}=0$, $x^{n_x -1} \neq 0$, and $x^{n_x -1} \neq x$ since $x^{n_x -1} = x$ implies $n_x =2$. Thus, \begin{align}
		x \to x^{n_x -1} \to x
	\end{align} is a cycle in $\overrightarrow{\Gamma}_E(S)$.
\end{proof}

\begin{corollary}\label{gggg}
	Let $S$ be a non-commutative semigroup with $0$. If $S$ contains a nilpotent $x$ such that $n_x \geq 3$, then $\overrightarrow{\Gamma}_E (S)$ contains a cycle. Moreover, $Girth(\overrightarrow{\Gamma}_E (S))=2$.
\end{corollary}
\begin{proof}
	Let $x \in Nil(S)$ such that $n_x \geq 3$. Then, $x x^{n_x -1}=0$, $x^{n_x -1} \neq 0$, and $x^{n_x -1} \neq x$ since $x^{n_x -1} = x$ implies $n_x =2$. Thus, \begin{align}
		x \to x^{n_x -1} \to x
	\end{align} is a cycle in $\overrightarrow{\Gamma}_E(S)$.
\end{proof}

\begin{theorem}
	Let $Z(S)=Nil(S) \neq \{0\}$. 
	\begin{enumerate}
		\item If $|Z(S) \setminus \{0\}|=1$, then $Girth(\overrightarrow{\Gamma}_E(S))=\infty$.
		\item If $|Z(S) \setminus \{0\}| \geq 2$, then $Girth(\overrightarrow{\Gamma}_E(S))=2$.
	\end{enumerate}
\end{theorem}

\begin{proof}
	\begin{enumerate}
		\item Since the graph has a single vertex and no edges, it contains no cycles. Hence,
		\begin{align}
			Girth(\overrightarrow{\Gamma}_E(S))=\infty.
		\end{align}
		
		\item Let $|Z(S) \setminus \{0\}| \geq 2$.
		\begin{enumerate}
			\item If there exists $x \in Z(S) \setminus \{0\}$ such that $n_x \geq 3$, then, by Corollary \ref{gggg}, the graph $\overrightarrow{\Gamma}_E(S)$ contains a cycle of length $2$.
			
			\item Assume that $n_s = 2$ for every $s \in Z(S) \setminus \{0\}$, and let $x,y \in Z(S) \setminus \{0\}$ be distinct. Then $x^2 = 0$ and $y^2 = 0$. Moreover,
			\begin{align}
				x = xy &\Rightarrow x = (xy)y = xy^2 =0 \label{316}\\
				x = yx &\Rightarrow x = y(yx) = xy^2 =0 \label{317}\\
				y = xy &\Rightarrow y = x(xy) = x^2y =0 \label{318} \\
				y = yx &\Rightarrow y = (yx)x = yx^2 = 0 \label{319}
			\end{align}
			Eqs.~\eqref{316}--\eqref{319} contradict that $x \neq0$ and $y \neq 0$ thus, $x \neq xy$, $x \neq yx$, $y \neq xy$, and $y \neq yx$ is obtained.
			
			\begin{enumerate}
				\item If $xy = yx =0$, then
				\begin{align}
					x \to y \to x
				\end{align}
				is a cycle in the digraph $\overrightarrow{\Gamma}_E(S)$.
				
				\item If $xy = 0$ and $yx \neq 0$, then
				\begin{align}
					x(yx) = (xy)x = 0 \quad \text{and} \quad (yx)x = yx^2 = 0.
				\end{align}
				Hence,
				\begin{align}
					x \to yx \to x
				\end{align}
				is a cycle in the digraph $\overrightarrow{\Gamma}_E(S)$.
				
				\item The case where $xy \neq 0$ and $yx = 0$ follows similarly.
				
				\item If $xy \neq 0$ and $yx \neq 0$, then
				\begin{enumerate}
					\item If $xy \neq yx$, then
					\begin{align}
						(xy)(yx) = xy^2x = 0 \quad \text{and} \quad (yx)(xy)=yx^2y=0,
					\end{align}
					so
					\begin{align}
						xy \to yx \to xy
					\end{align}
					is a cycle in the digraph $\overrightarrow{\Gamma}_E(S)$.
					
					\item If $xy = yx$, then
					\begin{align}
						x(xy)=x^2y=0 \quad \text{and} \quad x(yx)=x(xy)=x^2y = 0,
					\end{align}
					so
					\begin{align}
						x \to xy \to x
					\end{align}
					is a cycle in the digraph $\overrightarrow{\Gamma}_E(S)$.
				\end{enumerate}
			\end{enumerate}
		\end{enumerate}
	\end{enumerate}
\end{proof}

\section{Conclusion}
In this paper, $\overrightarrow{\Gamma}_E(S)$ for a semigroup $S$ is introduced and studied. $\overrightarrow{\Gamma}(S) \subseteq \overrightarrow{\Gamma}_E(S)$ is proved. The necessary and sufficient conditions for which $\overrightarrow{\Gamma}_E (S) = \overrightarrow{\Gamma}(S)$ are established. Properties of vertices that are adjacent to and adjacent from all other vertices are determined. The conditions in which $Diam(\overrightarrow{\Gamma}_E(S)) \leq Diam(\overrightarrow{\Gamma}(S))$ and $Girth(\overrightarrow{\Gamma}_E(S)) \leq Girth(\overrightarrow{\Gamma}(S))$ are identified by means of the existing properties of such vertices.

\section*{Acknowledgements}
This paper is derived from the first author's master's thesis supervised by the second author. The third author contributed to the development of further results and to the writing of the paper. All authors read and approved the final version of the paper.

\bibliographystyle{plain} 

\end{document}